\documentclass[11pt,reqno]{amsart}

\date{November 19, 2019}

\usepackage[parfill]{parskip}    
\usepackage{graphicx}
\usepackage{amssymb}
\usepackage{amsmath}
\usepackage{amsthm}
\usepackage{mathabx}
\usepackage{commath}
\usepackage{epstopdf}

\usepackage{amsmath,amsfonts,amsthm,amssymb,amsxtra,dsfont,mathrsfs}
\usepackage{hyperref} 
\usepackage{colordvi}
\usepackage[usenames,dvipsnames]{color}

\usepackage[usenames]{xcolor}
\usepackage{url}

\usepackage[T1]{fontenc}
\usepackage[utf8]{inputenc}

\usepackage[ngerman, german, english]{babel} 
\usepackage{bbm}

 \numberwithin{equation}{section}

\newtheorem{thm}{Theorem}[section]
\newtheorem{cor}[thm]{Corollary}
\newtheorem{lem}[thm]{Lemma}
\newtheorem{prop}[thm]{Proposition}

\theoremstyle{remark}

\theoremstyle{definition}

\newtheorem{assumption}[thm]{Assumption}

\newcommand{\R}{\mathbb{R}} 

\newcommand{\eps}{\epsilon}

\title[Energy asymptotics in the Brezis--Nirenberg problem. The higher-dimensional case]{Energy asymptotics in the Brezis--Nirenberg problem. \linebreak The higher-dimensional case}

\author{Rupert L. Frank}

\address[Rupert L. Frank]{Mathematisches Institut, Ludwig-Maximilians-Universit\"at M\"unchen, Theresienstr. 39, 80333 M\"unchen, Germany, and Mathematics 253-37, Caltech, Pasa\-de\-na, CA 91125, USA}

\email{r.frank@lmu.de, rlfrank@caltech.edu}

\author{Tobias König}

\address[Tobias König]{Mathematisches Institut, Ludwig-Maximilians-Universit\"at M\"unchen, Theresienstr. 39, 80333 M\"unchen, Germany}

\email{tkoenig@math.lmu.de}

\author {Hynek Kova\v{r}\'{\i}k}

\address [Hynek Kova\v{r}\'{\i}k]{DICATAM, Sezione di Matematica, Universit\`a degli studi di Brescia, Via Branze 38-  25123 Brescia, Italy}

\email {hynek.kovarik@unibs.it}

\thanks{\copyright\, 2019 by the authors. This paper may be reproduced, in its entirety, for non-commercial purposes.\\
Partial support through US National Science Foundation grant DMS-1363432 (R.L.F.) and Studienstiftung des deutschen Volkes (T.K.) is acknowledged. H.~K.  has been partially supported by Gruppo Nazionale per Analisi Matematica, la Probabilit\`a e le loro Applicazioni (GNAMPA) of the Istituto Nazionale di Alta Matematica (INdAM)}

\begin{document}

\begin{abstract}
For dimensions $N \geq 4$, we consider the Br\'ezis-Nirenberg variational problem of finding 
\[ S(\eps V) := \inf_{0\not\equiv u\in H^1_0(\Omega)} \frac{\int_\Omega |\nabla u|^2 \, dx +\eps \int_\Omega V\, |u|^2 \, dx}{\left(\int_\Omega |u|^q \, dx \right)^{2/q}}, \]
where $q=\frac{2N}{N-2}$ is the critical Sobolev exponent, $\Omega \subset \R^N$ is a bounded open set and $V:\overline{\Omega}\to \R$ is a continuous function. We compute the asymptotics of $S(0) - S(\eps V)$ to leading order as $\eps \to 0+$. We give a precise description of the blow-up profile of (almost) minimizing sequences and, in particular, we characterize the concentration points as being extrema of a quotient involving the Robin function. This complements the results from our recent paper in the case $N = 3$.
\end{abstract}

\maketitle

\section{\bf Introduction and main results}

\subsection{Setting of the problem}
Let $N \geq 4$ and let $\Omega \subset \R^N$ be a bounded open set. For $\eps > 0$ and a function $V \in C(\overline{\Omega})$, Br\'ezis and Nirenberg study in their famous paper \cite{BrNi} the quotient functional
\begin{equation} 
\label{var-prob}
\mathcal S_{\eps V}[u] := \frac{\int_\Omega |\nabla u|^2 \, dx +\eps \int_\Omega V\, |u|^2 \, dx}{\left(\int_\Omega |u|^q \, dx \right)^{2/q}},  \qquad q=\frac{2N}{N-2} \, ,
\end{equation}
and the corresponding variational problem of finding
\begin{equation}
\label{var-prob-inf}
S(\eps V) := \inf_{0\not\equiv u\in H^1_0(\Omega)} \mathcal S_{\eps V}[u] \,.
\end{equation}
This number is to be compared with
$$
 S_N = \pi N (N-2) \left(\frac{\Gamma(N/2)}{\Gamma(N)} \right)^{2/n}\, ,$$
the sharp constant \cite{Rod,Ro,Au,Ta} in the Sobolev inequality. Indeed, in \cite{BrNi} it is shown that $S(\eps V) < S_N$ as soon as  
\begin{equation} \label{N-def}
\mathcal N(V):= \{x\in \Omega: V(x) < 0\}
\end{equation}
is non-empty. This behavior is in stark contrast to the case $N = 3$ also treated in \cite{BrNi}, where there is an $\epsilon_V>0$ such that $S(\epsilon V) = S_N$ for all $\epsilon \in (0, \eps_V]$ even if $\mathcal N (V)$ is non-empty.

The purpose of this paper is, for $N \geq 4$, to describe the asymptotics of $S_N- S(\epsilon V)$ to leading order as $\epsilon \to 0$, as well as the asymptotic behavior of corresponding (almost) minimizing sequences and, in particular, their concentration behavior. This is the higher-dimensional complement to our recent paper \cite{FrKoKo}, where analogous results are shown in the more difficult case $N = 3$.

\textbf{Notation.   }
To prepare the statement of our main results, we now introduce some key objects for the following analysis. An important role is played by the Green's function of the Dirichlet Laplacian on $\Omega$, which in the sense of distributions satisfies, in the normalization of \cite{rey2},
\begin{equation} \label{G-pde}
\left\{
\begin{array}{l@{\quad}l}
-\Delta_x\, G(x,y) = (N-2)\, \omega_N\, \delta_y & \quad \text{in} \ \ \Omega , \\
& \\
G(x,y) = 0  & \quad \text{on} \ \ \partial\Omega,
\end{array}
\right.  
\end{equation}
where $\omega_N$ is the surface of the unit sphere in $\R^N$, and $\delta_y$ denotes the Dirac delta function centered at $y$. We denote by 
\begin{equation} \label{h-function}
H(x,y) = \frac{1}{|x-y|^{N-2}} - G(x,y)
\end{equation}
the regular part of $G$. The function $H(x, \cdot)$, defined on $\Omega \setminus \{x\}$, extends to a continuous function on $\Omega$ and we may define the \emph{Robin function} 
\begin{equation}
\label{phi-function}
\phi(x) := H(x,x) \,. 
\end{equation}
Using this function, we define the numbers
\begin{align*}
 \sigma_N(\Omega, V) & := \sup_{x\in\mathcal N(V)} \left( \phi(x)^{-\frac{2}{N-4}}\ |V(x)|^{\frac{N-2}{N-4}} \right), & N \geq 5  \, ,\\
 \sigma_4(\Omega, V) & := \sup_{x\in\mathcal N(V)} \left( \phi(x)^{-1} |V(x)|\right) ,  & N = 4 \,, 
\end{align*}
which will turn out to essentially be the coefficients of the leading order term in  $S_N- S(\epsilon V)$.

Another central role is played by the family of functions 
\begin{equation} \label{u-function}
U_{x,\lambda} (y) = \frac{\lambda^{(N-2)/2}}{(1+\lambda^2 |x-y|^2 )^{(N-2)/2}}\, \quad x\in \R^N, \, \lambda > 0.   
\end{equation}
It is well-known that the $U_{x, \lambda}$ are exactly the optimizers of the Sobolev inequality on $\R^N$.

Since \eqref{var-prob} is a perturbation of the Sobolev quotient, it is reasonable to expect the $U_{x,\lambda}$ to be nearly optimal functions for \eqref{var-prob-inf}. However, since \eqref{var-prob-inf} is set on $H^1_0(\Omega)$, we consider, as in \cite{BaCo}, the functions $PU_{x, \lambda} \in H^1_0(\Omega)$ uniquely determined by the properties 
\begin{equation} \label{eq-pu}
\Delta PU_{x,\lambda} = \Delta U_{x,\lambda}\ \ \  \text{ in } \Omega, \qquad PU_{x,\lambda} = 0 \ \ \ \text{ on } \partial \Omega \,.
\end{equation}
 
Moreover, let
$$
T_{x, \lambda} := \text{ span}\, \big\{ PU_{x, \lambda}, \partial_\lambda PU_{x, \lambda}, \partial_{x_i} PU_{x, \lambda}\,  (i=1,2,\ldots, N) \big\}
$$
and let $T_{x, \lambda}^\perp$ be the orthogonal complement of $T_{x,\lambda}$ in $H^1_0(\Omega)$ with respect to the inner product $\int_\Omega \nabla u \cdot \nabla v\,dy$.  

In what follows we denote by $\|\cdot\|$ the $L^2-$norm on $\Omega$. 
Finally, given a set $X$ and two functions $f_1,\, f_2: X\to\R$, we write $f_1 \lesssim f_2$ if there exists a numerical constant $c$ such that $f_1(x) \leq c\, f_2(x)$ for all $x\in X$. 


\subsection{Main results}

Throughout this paper and without further mention we assume that the following properties are satisfied.

\begin{assumption}
The set $\Omega \subset \R^N$, $N \geq 4$, is open and bounded and has a $C^2$ boundary. Moreover, $V \in C(\overline{\Omega})$ and $\mathcal N(V) \neq \emptyset$, with $\mathcal N(V)$ given by \eqref{N-def}.
\end{assumption}

Here is our first main result. It gives the asymptotics of $S_N - S(\eps V)$ to leading order in $\epsilon$. 

\begin{thm}
\label{thm expansion}
As $\eps\to 0+$, we have 
\begin{equation} \label{eq-thm ngeq5}
S(\eps V) = S_N -  C_N\,  \sigma_N(\Omega, V)\ \eps^{\frac{N-2}{N-4}} + o(\eps^{\frac{N-2}{N-4}}) \qquad\qquad\ \text{\rm if} \ N \geq 5 
\end{equation}
and 
\begin{equation} \label{eq-thm n4}
S(\eps V) =  S_4 -  \exp\Big( - \frac 4\epsilon \left(1 +o(1)\right) \sigma_4(\Omega, V)^{-1}  \Big) \qquad  \qquad \text{\rm if} \ N = 4.
\end{equation}
Here the constants $C_N$ are defined in \eqref{cn} below. 
\end{thm}

Our second main result shows that the blow-up profile of an arbitrary almost minimizing sequence $(u_\epsilon)$ is given to leading order by the family of functions $PU_{x, \lambda}$. Moreover, we give a precise characterization of the blow-up speed $\lambda = \lambda_\eps$ and of the point $x_0$ around which the $u_\eps$ concentrate. 

\begin{thm} 
\label{thm-minimizers}
Let $(u_\epsilon)\subset H^1_0(\Omega)$ be a family of functions such that
\begin{equation} \label{appr-min}
\lim_{\epsilon\to 0} \frac{\mathcal S_{\epsilon V}[u_\epsilon] - S(\epsilon V)}{S_N-S(\epsilon V)}  = 0 \qquad \text{and} \qquad \int_\Omega |u_\epsilon|^q \,dx = \left( \frac {S_N}{N(N-2)}\right)^{\frac{q}{q-2}} \,.
\end{equation}
Then there are $(x_\epsilon)\subset\Omega$, $(\lambda_\epsilon)\subset(0,\infty)$, $(\alpha_\epsilon)\subset\R$ and $(w_\eps) \subset H^1_0(\Omega)$ with $w_\eps \in T_{x_\eps, \lambda_\eps}^\perp$ such that
\begin{equation} \label{u-eps-final}
u_\epsilon =  \alpha_\epsilon \left( PU_{x_\epsilon, \lambda_\epsilon} + w_\eps \right)   
\end{equation}
and, along a subsequence, $x_\epsilon  \to x_0$ for some $x_0\in \mathcal N(V)$. Moreover, 
\begin{align*}
&  \begin{cases}  
  \phi(x_0)^{-\frac{2}{N-4}}\ |V(x_0)|^{\frac{N-2}{N-4}} = \sigma_N(\Omega, V) , &  N \geq 5,  \\
  \phi(x_0)^{-1}|V(x_0)|   = \sigma_4(\Omega, V) \,,  &  N = 4,
  \end{cases} \\
  & \begin{cases}
\|\nabla w_\eps\|=o(\epsilon^\frac{N-2}{2N-8}),  & N \geq 5, \\
\|\nabla w_\eps\|\leq  \exp\Big( - \frac 2\epsilon \left(1 +o(1)\right) \sigma_4(\Omega, V)^{-1} \Big) , & N = 4, 
\end{cases} \\
& \begin{cases}
 \lim_{\epsilon \to 0}\,  \epsilon \,  \lambda_\epsilon^{N-4}  = \frac{N\, (N-2)^2\, a_N \,\phi(x_0)}{2 \,b_N\,  |V(x_0)|} \,, & N \geq 5, \\
 \lim_{\epsilon \to 0}\,  \epsilon \, \ln \lambda_\eps =  \frac{2\, \phi(x_0)}{  |V(x_0)|}\, , & N = 4,
 \end{cases}  \\
& \begin{cases} 
\alpha_\epsilon = s \left( 1 + D_N \epsilon^\frac{N-2}{N-4} + o(\epsilon^\frac{N-2}{N-4}) \right), & N \geq 5,  \\
\alpha_\eps = s \left(1 + \exp\Big( - \frac 4\epsilon \left(1 +o(1)\right) \sigma_4(\Omega, V)^{-1}  \Big) \right) , & N=4,
\end{cases} 
\qquad\text{for some}\ s\in\{\pm 1\} \,. 
\end{align*}
Here the constants $a_N$, $b_N$ and $D_N$ are defined in \eqref{anbn} and \eqref{dn} below. 
\end{thm}

The coefficients appearing in Theorems \ref{thm expansion} and \ref{thm-minimizers} are
\begin{equation}
\label{anbn}
a_N := \int_{\R^N} \frac{ dz}{(1+z^2)^{(N+2)/2}}, \qquad b_N := \begin{cases} \int_{\R^N} \frac{ dz}{(1+z^2)^{N-2}}, & N \geq 5, \\
\omega_4, & N = 4, \end{cases}
\end{equation}
as well as 
\begin{equation} \label{cn}
C_N := S_N^{\frac{2-N}{2}}\, (N(N-2))^{\frac{N-2}{2}}\ \frac{N-4}{N-2} \, \left(\frac{N (N-2)^2}{2}\right)^{\frac{2}{4-N}}\, a_N^{-\frac{2}{N-4}}\, b_N^{\frac{N-2}{N-4}}, \qquad N \geq 5, 
\end{equation}
and 
\begin{equation}
\label{dn}
D_N :=   a_N ^{-\frac{2}{N-4}}  b_N ^{\frac{N-2}{N-4}} S_N^{-\frac{N}{2}} \left( N(N-2) \right)^{\frac{N}{2} - \frac{N-2}{N-4}} \left( \frac{N-2}{2} \right)^{-\frac{N-2}{N-4}}\,, \qquad N \geq 5. 
\end{equation}
A simple computation using beta functions yields the numerical values 
\[ a_N = \frac{\omega_N}{N}, \quad N \geq 4, \quad \quad \text{ and } \quad \quad  b_N =   \omega_N \, \frac{\Gamma\left(\frac N2\right)\, \Gamma\left(\frac N2-2\right)}{2\, \Gamma(N-2)}, \quad  N \geq 5. \]


\subsection{Discussion}

Let us put our main results, Theorems \ref{thm expansion} and \ref{thm-minimizers}, into perspective with respect to existing results in the literature. 

Of course, minimizers of the variational problem \eqref{var-prob-inf} satisfy the corresponding Euler--Lagrange equation. It is natural to study general positive solutions of this equation, even if they do not arise as minimizers of \eqref{var-prob-inf}. In the special case where $V$ is a negative constant, Br\'ezis and Peletier \cite{BrPe} discussed the concentration behavior of such general solutions and made some conjectures, which were later proved by Han \cite{Ha} and Rey \cite{rey1}. Probably one can use their precise concentration results to give an alternative proof of our main results in the special case where $V$ is constant and probably one can even extend the analysis of Han and Rey to the case of non-constant $V$.

Our approach here is different and, we believe, simpler for the problem at hand. We work directly with the variational problem \eqref{var-prob-inf} and \emph{not} with the Euler--Lagrange equation. Therefore, our concentration results are not only true for minimizers but even for `almost minimizers' in the sense of \eqref{appr-min}. We believe that this is interesting in its own right. On the other hand, a disadvantage of our method compared to the Han--Rey method is that it gives concentration results only in $H^1$ norm and not in $L^\infty$ norm and that it is restricted to energy minimizing solutions of the Euler--Lagrange equation.

In the special case where $V$ is a negative constant, our results are very similar to results obtained by Takahashi \cite{Tak}, who combined elements from the Han--Rey analysis (see, e.g., \cite[Equation (2.4) and Lemma 2.6]{Tak}) with variational ideas adapted from Wei's treatment \cite{We} of a closely related problem; see also \cite{FlWe}. Takahashi obtains the energy asymptotics in Theorem \ref{thm expansion} as well as the characterization of the concentration point and the concentration scale in Theorem \ref{thm-minimizers} under the assumption that $u_\epsilon$ is a minimizer for \eqref{var-prob-inf}. Thus, in our paper we generalize Takahashi's results to non-constant $V$ and to almost minimizing sequences and we give an alternative, self-contained proof which does not rely on the works of Han and Rey.

The present work is a companion paper to \cite{FrKoKo} relying on the techniques developed there in the three dimensional case. In particular, Theorems \ref{thm expansion} and \ref{thm-minimizers} should be compared with \cite[Theorems 1.3 and 1.7]{FrKoKo}, respectively. Although the expansions for $N \geq 4$ have the same structure as in the case $N = 3$, the latter case is more involved. In fact, when $N = 3$, the coefficient of the leading order term, namely the term of order $\eps$, vanishes and one has to expand the energy to the next order, namely $\eps^2$.

Besides the extensions of known results that we achieve here, we also think it is worthwhile from a methodological point of view to present our arguments again in the conceptually easier case $N\geq 4$. In the three-dimensional case the basic technique is iterated twice, which to some extent obscures the underlying simple idea. Moreover, we hope our work sheds some new light on the similarities and discrepancies between the two cases. 

The structure of this paper is as follows. In Section \ref{sec upperbd} we prove the upper bound from Theorem \ref{thm expansion} by inserting the $PU_{x, \lambda}$ as test functions. The proof of the corresponding lower bound is prepared in Sections \ref{sec lowerbd pre} and \ref{sec lowerbd exp}, where we derive a crude asymptotic expansion for a general almost minimizing sequence $(u_\eps)$ and the corresponding expansion of $\mathcal S_{\eps V}[u_\eps]$. Section \ref{sec pfmain} contains the proof of Theorems \ref{thm expansion} and \ref{thm-minimizers}. A crucial ingredient there is the coercivity inequality \eqref{eq-rey} from \cite{rey2}, which allows us to estimate the remainder terms and to refine the aforementioned expansion of $u_\eps$. Finally, an appendix contains two auxiliary technical results. 

\section{\bf Upper bound}
\label{sec upperbd}

The computation of the upper bound to $S(\eps V)$ uses the functions $PU_{x, \lambda}$, with suitably chosen $x$ and $\lambda$, as test functions. The following theorem gives a precise expansion of the value $\mathcal S_{\eps V}[PU_{x, \lambda}]$. To state it, we introduce the distance to the boundary of $\Omega$ as 
$$
d(x) = \text{dist}(x,\partial\Omega), \qquad x\in\Omega.
$$

\begin{thm}
\label{thm expansion PU}
Let $x = x_\lambda$ be a sequence of points such that $d(x) \lambda \to \infty$. Then as $\lambda \to \infty$, we have
\begin{align}
 \int_\Omega  | \nabla PU_{x, \lambda}|^2 \, dy &= 
 N(N-2)  \left(\frac{S_N}{N(N-2)} \right) ^\frac{q}{q-2} + N(N-2) \, a_N\, \lambda^{2-N}\, \phi(x) + \mathcal O((d(x)\lambda)^{\frac{4}{3}-N})\, ,  \label{exp-NablaPU} \\
 \int_\Omega  V PU_{x, \lambda}^2 \, dy  &= \begin{cases} 
\lambda^{-2}\, b_N\, V(x) + \mathcal{O}\left((d(x) \lambda)^{2-N} \right) + o(\lambda^{-2})  , & N \geq 5, \label{exp-epsVPU}\\
\frac{\log \lambda}{\lambda^2}\ b_4\, V(x) + \mathcal{O}\left((d(x) \lambda)^{-2}\right) + o\left(\frac{\log \lambda}{\lambda^2}\right) & N = 4,
\end{cases}
\end{align}
and
\begin{equation}
\label{exp-PUq}
\int_\Omega | PU_{x, \lambda}|^q \, dy =\left( \frac{S_N}{N(N-2)} \right) ^\frac{q}{q-2} - q\, a_N\, \lambda^{2-N}\, \phi(x) + o ((d(x) \lambda)^{2-N}).
\end{equation}
\smallskip

\noindent In particular, as $\lambda \to \infty$, 
\begin{equation}
\label{exp-quot-PU}
\mathcal S_{\eps V}[PU_{x, \lambda}] = 
\begin{cases}
S_N \!+ \!\left( \!\frac{S_N}{N(N-2)}\!\right)^{\frac{2}{2-q}}\!\! \left( \!\frac{N(N-2)\, a_N \, \phi(x)}{\lambda^{N-2}} +b_N\, \eps\, \frac{V(x)}{\lambda^{2}} \!\right) +  o ((d(x)\lambda)^{2-N}) + o (\eps \lambda^{-2}) , & N \geq 5, \\
S_4 + \frac {8}{S_4} \left( \frac{8\, a_4 \, \phi(x)}{\lambda^{2}} +b_4\, \eps\, \frac{V(x)\, \log\lambda}{\lambda^{2}} \right) + o ((d(x)\lambda)^{-2}) + o (\eps \frac{\log \lambda}{\lambda^2}), & N = 4. 
\end{cases}
\end{equation}
\end{thm}

In view of Proposition \ref{prop-app-min} below, the assumption $d(x) \lambda \to \infty$ in Theorem \ref{thm expansion PU} is no restriction, even when dealing with general almost minimizing sequences.

\begin{cor}
\label{cor-upperb}
As $\eps\to 0+$, we have 
\begin{equation} \label{eq-upperb1}
S(\eps V)  \leq S_N -  C_N\,  \sigma_N(\Omega, V)\ \eps^{\frac{N-2}{N-4}} + o(\eps^{\frac{N-2}{N-4}}) \qquad\qquad\ \text{\rm if} \ N \geq 5 
\end{equation}
and 
\begin{equation} \label{eq-upperb2}
S(\eps V) \leq S_4 -  \exp\Big( - \frac 4\epsilon \left(1 +o(1)\right) \sigma_4(\Omega, V)^{-1} \Big) \qquad  \qquad \text{\rm if} \ N = 4.
\end{equation}
\end{cor}

\begin{proof}
[Proof of Corollary \ref{cor-upperb}]
By \cite[(2.8)]{rey2}, we have 
\begin{equation}
\label{phi near bdry}
d(x)^{2-N} \lesssim \phi(x) \lesssim d(x) ^{2-N}.
\end{equation}
(Note that this bound uses the $C^2$ assumption on $\Omega$.) 
Since, moreover, $V=0$ on $\partial\mathcal N(V)\setminus\partial\Omega$, the function $\phi^{-\frac{2}{N-4}}\ |V|^{\frac{N-2}{N-4}}$ can be extended to a continuous function on $\overline{\mathcal N(V)}$ which vanishes on $\partial\mathcal N(V)$. Thus there is $z_0 \in \mathcal N(V)$ such that
\begin{equation} \label{z0-1}
\sigma_N(\Omega, V) = \phi(z_0)^{-\frac{2}{N-4}}\ |V(z_0)|^{\frac{N-2}{N-4}}, \qquad N \geq 5 .
\end{equation}
The corollary for $N \geq 5$ now follows by choosing $x = z_0$ in \eqref{exp-quot-PU} and optimizing the quantity $\frac{N(N-2)\, a_N \, \phi(z_0)}{\lambda^{N-2}} +b_N\, \eps\, \frac{V(z_0)}{\lambda^{2}}$ in $\lambda$. The optimal choice is 
\begin{equation} \label{lambda-eps}
\lambda(\eps) = \left(\frac{N\, (N-2)^2\, a_N \,\phi(z_0)}{2 \,b_N\,  |V(z_0)|}\right)^{\frac{1}{N-4}}\, \eps^{-\frac{1}{N-4}} \, ,
\end{equation}
and \eqref{eq-upperb1} follows from a straightforward computation.

Similarly, if $N = 4$, since $\frac{\phi(y)}{|V(y)|}$ is a positive continuous function on $\mathcal N(V)$ which goes to $+\infty$ as $y \to \partial \mathcal N(V)$, we find some $z_0 \in \mathcal N(V)$ such that 
\begin{equation} \label{z0-2}
\sigma_4(\Omega, V) = \frac{\phi(z_0)}{|V(z_0)|} \, .
\end{equation}
Thus we may choose $x = z_0$ in \eqref{exp-quot-PU} and optimize the quantity $A \lambda^{-2} - B \epsilon \lambda^{-2} \log \lambda$ in $\lambda > 0$, where $A = 8\, a_4 \, \phi(z_0) + o(1)$ and $B = b_4\, |V(z_0)| + o(1)$. The optimal choice is
\begin{equation} \label{lambda-eps-2}
\lambda(\eps) = \sqrt{e} \exp  \left(\frac{A}{B \epsilon} \right).
\end{equation}
Inserting this into \eqref{exp-quot-PU}, we get 
\begin{align*}
S(\epsilon V) \leq \mathcal S_{\epsilon V}[PU_{x, \lambda(\epsilon)}] &= S_4 - \frac{4 b_4}{e S_4} \epsilon |V(z_0)| \exp\left(- \frac{16\, a_4 \, (\phi(z_0) + o(1))}{b_4\, \eps\, |V(z_0)| + o(1)}\right)  \\
&= S_4 -  \exp\Big( - \frac 4\epsilon \left(1 +o(1)\right) \inf_{x\in\mathcal N(V)} \frac{\phi(x)}{|V(x)|}  \Big)\,, 
\end{align*}
where we have used the fact that 
\begin{equation} \label{eq-revised}
\epsilon\, b \exp\Big(-\frac a\epsilon\Big) = \exp\Big(-\frac a\epsilon +o\Big(\frac 1\epsilon \Big)\Big) , \qquad \epsilon\to 0+
\end{equation}
holds for all $a\geq 0$ and all $b>0$.

This completes the proof of \eqref{eq-upperb2}, and thus of Corollary \ref{cor-upperb}.
\end{proof}

\begin{proof}
[Proof of Theorem \ref{thm expansion PU}]
We prove equations \eqref{exp-NablaPU}--\eqref{exp-PUq} separately. Then expansion \eqref{exp-quot-PU} follows by a straightforward Taylor expansion of the quotient functional $\mathcal S_{\eps V}[PU_{x, \lambda}]$. 

\emph{Proof of \eqref{exp-NablaPU}.    } 
Since the $U_{x, \lambda}$ satisfy the equation 
\begin{equation} \label{u-eq}
-\Delta_y U_{x,\lambda}(y) = N(N-2)\, U_{x,\lambda}(y)^{q-1}, \quad y \in \R^N,
\end{equation}
it follows using integration by parts that 
$$
\int_\Omega | \nabla P U_{x, \lambda}|^2 \, dy = N(N-2) \int_\Omega U^{q-1}_{x, \lambda} P U_{x, \lambda}\, dy .
$$
On the other hand, by \cite[Prop.~1]{rey2} we know that 
\begin{equation} \label{u-split}
P U_{x, \lambda} = U_{x, \lambda} -\varphi_{x,\lambda}, \qquad \varphi_{x,\lambda} = \frac{H(x,\cdot)}{\lambda^{(N-2)/2}} + f_{x,\lambda},
\end{equation}
where
\begin{equation} \label{sup-f}
\|f_{x,\lambda}\|_{L^\infty(\Omega)} = \mathcal{O} \left( \lambda^{-(N+2)/2}\,  d(x)^{-N}\right) , \qquad \lambda\to\infty.
\end{equation}
By putting the above equations together we obtain 
\begin{equation} \label{eq-1}
\int_\Omega | \nabla PU_{x, \lambda}|^2 \, dy = N(N-2) \left( \int_\Omega U^{q}_{x, \lambda} \, dy -\lambda^{\frac{2-N}{2}}  \int_\Omega U^{q-1}_{x, \lambda} \, H(x, \cdot) \, dy - \int_\Omega U^{q-1}_{x, \lambda}\,  f_{x,\lambda} \, dy \, \right).
\end{equation}

A direct calculation shows that 
\begin{equation} \label{uq}
\int_\Omega  U_{x, \lambda}^q \, dy = \int_{\R^N}  U_{x, \lambda}^q \, dy + \mathcal{O}((d(x)\lambda)^{-N}) = \left( \frac{S_N}{N(N-2)}\right)^{\frac{q}{q-2}} + \mathcal{O}((d(x)\lambda)^{-N}).
\end{equation}

Moreover, for any $x\in\Omega$ we have
\begin{equation}\label{sup-h}
d(x)^{2-N} \ \lesssim\ \| H(x, \cdot) \|_{L^\infty(\Omega)} \ \lesssim\ d(x)^{2-N}
\end{equation}
and 
\begin{equation}\label{sup-h'}
\sup_{y\in\Omega} | \nabla_y H(x, y) | \ \lesssim\  d(x)^{1-N},
\end{equation}
see \cite[Sec.~2 and Appendix]{rey2}. 

Now let $\rho \in (0, \frac{d(x)}{2})$. A direct calculation using \eqref{u-function}, \eqref{sup-h} and \eqref{sup-h'} shows that 
\begin{align*}
\int_{B_\rho(x)} U^{q-1}_{x, \lambda} \, H(x, \cdot) \, dy &= \lambda^{1+\frac N2}\, \big(\phi(x)+\mathcal{O}(\rho\, d(x)^{1-N})\big) 
 \int_{B_\rho(x)}\frac{dy}{(1+\lambda^2|x-y|^2)^{(N+2)/2} } \\
 & = \lambda^{1-\frac N2}\,  a_N\, \left(\phi(x)+\mathcal{O}(\rho\, d(x)^{1-N})\right) ( 1+\mathcal{O}((\lambda\, \rho)^{-2}))  
\end{align*}
and
\begin{align*}
\int_{\Omega \setminus B_\rho(x)} U^{q-1}_{x, \lambda} \, H(x, \cdot) \, dy &= \lambda^{1+\frac N2}\, \mathcal{O}(d(x)^{2-N}) \int_\rho^\infty \frac{r^{N-1}\, dr}{(1+\lambda^2\, r^2)^{\frac{N+2}{2}}} \\
& = \lambda^{1-\frac N2}\, \mathcal{O}(d(x)^{2-N}) \int_{\rho\lambda}^\infty \frac{t^{N-1}\, dt}{(1+ t^2)^{\frac{N+2}{2}}} \\
& =  \lambda^{1-\frac N2}\,  \mathcal{O}\left(d(x)^{2-N}\, (\lambda\, \rho)^{-2}\right). 
\end{align*}
Hence for the second term on the right hand side of \eqref{eq-1} we get
\begin{align} \label{u5h}
\lambda^{\frac{2-N}{2}} \int_\Omega U^{q-1}_{x, \lambda} \, H(x, \cdot) \, dy & = \lambda^{2-N}\, a_N\, \phi(x)  + \lambda^{2-N} \, \mathcal{O}\left(\rho \, d(x)^{1-N}\right) + \lambda^{2-N} \,  \mathcal{O}\left(d(x)^{2-N}\, (\lambda\, \rho)^{-2}\right).
\end{align}
As for the last term on the right hand side of \eqref{eq-1}, we note that in view of \eqref{sup-f} 
$$
\left |  \int_\Omega U^{q-1}_{x, \lambda}\,  f_{x,\lambda} \, dy \, \right | \, \leq \|f_{x,\lambda}\|_{L^\infty(\Omega)}  \int_{\R^N} U^{q-1}_{x, \lambda} \, dy= \|f_{x,\lambda}\|_{L^\infty(\Omega)} \, a_N\, \lambda^{1-\frac N2} = 
 \mathcal{O}\left( (\lambda\, d(x))^{-N}\right).
$$
The claim thus follows from \eqref{eq-1} by choosing $\rho = d(x)^{1/3} \lambda^{-2/3}$ in \eqref{u5h}. (Notice that $\rho = d(x) (d(x) \lambda)^{-2/3}) \leq \frac{d(x)}{2}$ for $\lambda$ large enough.)

\emph{Proof of \eqref{exp-epsVPU}.    }
We have 
\begin{equation} \label{eq-V}
\int_\Omega V\, PU_{x, \lambda}^2 \, dy = \int_\Omega V\,  U_{x, \lambda}^2 \, dy + \int_\Omega V\, (\varphi_{x, \lambda}^2- 2\ U_{x, \lambda}\, \varphi_{x, \lambda} )\, dy \, .
\end{equation}
Since by \cite[Prop.~1]{rey2},
\begin{equation}\label{eq-rey-2}
0  \, \leq \, \varphi_{x,\lambda}(y) \, \leq\,  U_{x,\lambda}(y) \qquad \forall\ y\in\Omega,
\end{equation}
together with \eqref{u-split}, \eqref{sup-f} and \eqref{sup-h} we obtain the following upper bound on the last integral in \eqref{eq-V},
\begin{align*}
\Big | \int_\Omega V\, (\varphi_{x, \lambda}^2- 2\ U_{x, \lambda}\, \varphi_{x, \lambda} ) \, dy \, \Big| & \leq 2\, \|V\|_{L^\infty(\Omega)}\,  \|\varphi_{x, \lambda}\|_{L^\infty(\Omega)}  \int_\Omega U_{x, \lambda} \, dy  = \mathcal{O}\left((d(x) \lambda)^{2-N} \right)  \, .
\end{align*}
To treat the first term on the right hand side of \eqref{eq-V}, first assume $N \geq 5$. Choose a sequence $\rho = \rho_\lambda$ such that $\rho \leq d(x)$, $\rho \to 0$ and $\rho \lambda \to \infty$ as $\lambda \to \infty$. (This is always possible, whether or not $d \to 0$.) Then, by continuity of $V$, 
\begin{align*}
\int_\Omega V\,  U_{x, \lambda}^2 \, dy &=  (V(x)+ o(1)) \int_{B_\rho(x)} U_{x, \lambda}^2 \, dy + \int_{\Omega \setminus B_\rho(x)} V \, U_{x, \lambda}^2 \, dy \\
&=  \lambda^{-2}\, b_N\, V(x) + o(\lambda^{-2}) + \mathcal O \left( \int_{\Omega \setminus B_\rho(x)} U_{x, \lambda}^2 \, dy \right) \\
& =  \lambda^{-2}\, b_N\, V(x) + o(\lambda^{-2}) + \mathcal O \left(\lambda^{-2} (\rho \lambda)^{-N +4} \right) = \lambda^{-2}\, b_N\, V(x) + o(\lambda^{-2}). 
\end{align*}
Similarly, in the case $N = 4$ we let $B_\tau(x)$ and $B_R(x)$ be two balls centered at $x$ with radii $\tau$ and $R$ chosen such that $B_\tau(x) \subset \Omega \subset B_R(x)$ and split the last integration in two parts as follows. Extending $V$ by zero to $B_R(x) \setminus \Omega$ we get
\begin{align}
\int_{\Omega\setminus B_\tau(x)} V\,  U_{x, \lambda}^2 \, dy & = \int_{B_R(x)\setminus B_\tau(x)} V\,  U_{x, \lambda}^2 \, dy \leq \, \omega_4 \, \|V\|_{L^\infty(\Omega)}  \int_{\tau}^R \frac{\lambda^2}{(1+\lambda^2 |x-y|^2)^2}\ r^3\, dr \nonumber \\
& =  \omega_4 \, \|V\|_{L^\infty(\Omega)}\, \lambda^{-2}  \int_{\tau\lambda}^{R\lambda} \frac{t^3}{(1+t^2)^2}\ \, dt = \mathcal{O}(\lambda^{-2}\, \log(R/\tau)) \label{compl}.
\end{align}
On the other hand, denoting by $o_\tau(1)$ a quantity that vanishes as $\tau\to 0$ and assuming that $\tau\lambda\to \infty$ we get
\begin{align*}
\int_{B_\tau(x)} V\,  U_{x, \lambda}^2  \, dy & =b_4\,  V(x)\ \int_0^{\tau} \frac{\lambda^2\, r^3\, dr}{(1+\lambda^2 |x-y|^2)^2}\  + o_\tau(1)\, \int_0^{\tau} \frac{\lambda^2\, r^3\, dr}{(1+\lambda^2 |x-y|^2)^2} \\
& = b_4 \, \lambda^{-2}\, V(x) \int_0^{\tau\lambda} \frac{ t^3\, dt}{(1+t^2)^2} + \lambda^{-2}\, o_\tau(1)\, \int_0^{\tau\lambda} \frac{ t^3\, dt}{(1+t^2)^2} \\
& = b_4\, \frac{\log \lambda}{\lambda^2}\  V(x)  + o_\tau(1)\, \mathcal{O}\left(\frac{\log \lambda}{\lambda^2}\right) + \mathcal{O}\left(\frac{\log \tau}{\lambda^2}\right).
\end{align*}
By choosing $\tau= \frac{1}{\log \lambda}$ and taking into account \eqref{compl} we arrive at \eqref{exp-epsVPU} in case $N = 4$. 

\emph{Proof of \eqref{exp-PUq}.    }

Recall that $q>2$. 
Hence from the Taylor expansion of the function $t\mapsto t^q$ on an interval $[0, b]$ it follows that for any $a\in [0,b]$ we have 
\begin{equation} \label{taylor}
| \, b^q -(b-a)^q -q\, b^{q-1}\, a\, | \, \leq \frac{q (q-1)}{2}\ b^{q-2}\, a^2.
\end{equation}
Because of \eqref{eq-rey-2} and \eqref{u-split} we can apply \eqref{taylor} with $b=U_{x,\lambda}(y)$ and $a= \varphi_{x,\lambda}(y)$ to obtain the following point-wise upper bound:
\begin{align}
\label{taylor-1}
 \big |\, PU_{x, \lambda} ^q - U_{x,\lambda}^q + q\, U_{x,\lambda}^{q-1}\, \varphi_{x,\lambda}\, \, \big |  &\  \leq\  \frac{q(q-1)}{2}\ U_{x,\lambda}^{q-2}\,  \varphi_{x,\lambda}^2 
\end{align} 
Together with estimate \eqref{eq-b} this gives
\begin{align}
\left | \int_\Omega  \left( PU_{x, \lambda} ^q - U_{x,\lambda}^q + q\, U_{x,\lambda}^{q-1}\, \varphi_{x,\lambda} \right) \, dy \, \right | & =  \mathcal{O}\left((d(x)\, \lambda )^{-N}\right)\, . \label{taylor-2}
\end{align}
On the other hand, the calculations in the proof of \eqref{exp-NablaPU} show that 
$$
\int_\Omega U_{x,\lambda}^{q-1}\, \varphi_{x,\lambda} \, dy = \lambda^{2-N}\, a_N\,  \phi(x)  + \mathcal{O}\left((d(x)\lambda)^{\frac43 -N} \right) = \lambda^{2-N}\, a_N\,  \phi(x) + o \left((d(x)\lambda)^{2 -N} \right).
$$
In view of \eqref{uq} and \eqref{taylor-2} this completes the proof.
\end{proof}

\section{\bf Lower bound. Preliminaries }
\label{sec lowerbd pre}

\noindent As a starting point for the proof of the lower bound on $S(\epsilon V)$, we derive a crude asymptotic form of almost minimizers of $\mathcal S_{\eps V}$. The following result is essentially well-known. We have recalled the proof in \cite[Appendix B]{FrKoKo} in the case $N = 3$, but the same argument carries over to $N \geq 4$.

\begin{prop} \label{prop-app-min}
Let $(u_\epsilon)\subset H^1_0(\Omega)$ be a sequence of functions satisfying
\begin{equation}
\label{eq:appminass}
\mathcal S_{\epsilon V}[u_\epsilon] = S_N+o(1) \,,
\qquad
\int_\Omega |u_\epsilon|^q\,dx = \left( \frac{S_N}{N(N-2)} \right) ^\frac{q}{q-2} \,.
\end{equation}
Then, along a subsequence,
\begin{equation} \label{u-rey}
u_\epsilon = \alpha_\epsilon \left( PU_{x_\epsilon,\lambda_\epsilon} + w_\epsilon \right) ,
\end{equation}
where
\begin{equation}
\begin{aligned}  \label{lim-eps}
\alpha_\epsilon & \to s \qquad \text{for some}\ s\in\{-1,+1\} \,,\\
x_\epsilon & \to x_0 \qquad\text{for some}\ x_0\in\overline\Omega \,, \\
\lambda_\epsilon d_\epsilon & \to \infty \,,\\
\| \nabla w_\epsilon \| &\to 0 \qquad\text{and}\qquad w_\epsilon \in T_{x_\epsilon,\lambda_\epsilon}^\bot \,.  
\end{aligned}
\end{equation}
Here $d_\eps= \text{dist }(x_\eps,\partial\Omega)$. 
\end{prop}

\subsection*{Convention}
From now on we will assume that
\begin{equation}
\label{appr-min0}
S(\epsilon V)< S_N
\qquad\text{for all}\ \epsilon>0
\end{equation}
and that $(u_\epsilon)$ satisfies \eqref{appr-min}. In particular, assumption \eqref{eq:appminass} is satisfied. We will always work with a sequence of $\epsilon$'s for which the conclusions of Proposition \ref{prop-app-min} hold. To enhance readability, we will drop the index $\epsilon$ from $\alpha_\epsilon$, $x_\epsilon$, $\lambda_\epsilon$, $d_\epsilon$ and $w_\epsilon$.

\section{\bf Lower bound. The main expansion } 
\label{sec lowerbd exp}

In this section we expand $\mathcal S_{\eps V}[u_\eps]$ by using the decomposition \eqref{u-rey} of $u_\epsilon$. We shall show the following result.

\begin{prop}
\label{prop-lowerbd}
Let $(u_\epsilon)\subset H^1_0(\Omega)$ satisfy \eqref{u-rey} and \eqref{lim-eps}.
Then 
\begin{align}
|\alpha|^{-2} \int_\Omega |\nabla u_\eps|^2 \, dy & =  \int_\Omega |\nabla P U_{x, \lambda}|^2 \, dy + \int_\Omega |\nabla w|^2 \, dy \, , \label{str-1} \\
 |\alpha|^{-q} \int_\Omega |u_\eps|^q \, dy & = \int_\Omega  P U_{x, \lambda}^q \, dy +\frac{q(q-1)}{2}\, \int_\Omega U_{x, \lambda}^{q-2}\, w^2 \, dy + o\left(\int_\Omega |\nabla w|^2+ (\lambda d)^{2-N}\right) \, , \label{str-2} \\ 
|\alpha|^{-2} \eps \int_\Omega V u_\eps^2 \, dy & = \eps \int_\Omega V P U_{x, \lambda}^2  \, dy + \mathcal{O}\left(\eps \int_\Omega |\nabla w|^2 \, dy + \eps \sqrt{\int_\Omega |\nabla w|^2 \, dy}\, \sqrt{ \int_\Omega |V|\,  P U_{x, \lambda}^2 \, dy } \right)  .          \label{str-3}
\end{align}

In particular, 
\begin{align}
\mathcal S_{\eps V}[u_\eps]  &=  \mathcal S_{\eps V}[PU_{x, \lambda}] + I[w] + \mathcal{O}\left(\eps\,  \sqrt{\int_\Omega |\nabla w|^2 \, dy}\ \sqrt{ \int_\Omega |V| P U_{x, \lambda}^2 \, dy} \right) \nonumber \\
& \quad  +  o\left(\int_\Omega |\nabla w|^2 \, dy + (\lambda d)^{2-N}\right), \label{exp-quot-ueps}
\end{align}
where
\begin{equation}
\label{definition I}
I[w] :=  \left(\int_{\Omega} U_{x, \lambda}^q \, dy \right)^{-\frac 2q} \left( \int_\Omega |\nabla w|^2 \, dy - N(N+2)  \int_\Omega U_{x, \lambda}^{q-2}\, w^2 \, dy \right)  . 
\end{equation}
\end{prop}

\begin{proof}
We prove equations \eqref{str-1}--\eqref{str-3} separately. Then the expansion \eqref{exp-quot-ueps} follows by a straightforward Taylor expansion of the quotient functional $\mathcal S_{\eps V}$, using $\mathcal S_{\eps V}[u_\eps] =\mathcal S_{\eps V}[|\alpha|^{-1} u_\eps]$.

In the sequel we denote by $c_1, c_2, \dots $ various positive constants which are independent of $\eps$.

\emph{Proof of \eqref{str-1}. }  This follows by \eqref{u-rey} and  $w \in T_{x,\lambda}^\bot$.

\emph{Proof of \eqref{str-2}.    } Recall that $\alpha^{-1} u_\eps = U_{x, \lambda} + (w - \varphi_{x, \lambda})$ by \eqref{u-split} and \eqref{u-rey}. We use the associated pointwise estimate
\begin{align*}
& \left | |\alpha|^{-q} |u_\eps|^q - U_{x, \lambda}^q - q \,  U_{x, \lambda}^{q-1} (w- \varphi_{x, \lambda}) -\frac{q(q-1)}{2}\, U_{x, \lambda}^{q-2} (w-\varphi_{x, \lambda})^2 \right |   \\
& \qquad \leq c_1 \left(  |w-\varphi_{x, \lambda}|^q +  |w- \varphi_{x, \lambda}|^{q-(q-3)_+}  U_{x, \lambda}^{(q-3)_+} \right),
\end{align*}
where $(q-3)_+= \max\{q-3, 0\}$.
Using \eqref{taylor-1}, it follows that 
\begin{align*}
& \qquad \left | |\alpha|^{-q} |u_\eps|^q - PU_{x, \lambda}^q - q \,  U_{x, \lambda}^{q-1} w -\frac{q(q-1)}{2}\, U_{x, \lambda}^{q-2} w^2 \right | \\
&  \leq \, c_2 \left ( |w- \varphi_{x, \lambda}|^q +  |w-  \varphi_{x, \lambda}|^{q-(q-3)_+}  U_{x, \lambda}^{(q-3)_+} +   U_{x, \lambda}^{q-2} \varphi_{x, \lambda}\, |w| +  U_{x, \lambda}^{q-2} \varphi_{x, \lambda}^2 \right) \\
& \leq \, c_3 \left( |w|^q +\varphi_{x, \lambda}^q  +  |w|^{q-(q-3)_+}  U_{x, \lambda}^{(q-3)_+}  + \varphi_{x, \lambda}^{q-(q-3)_+}  U_{x, \lambda}^{(q-3)_+} + U_{x, \lambda}^{q-2} \varphi_{x, \lambda}\, |w| +  U_{x, \lambda}^{q-2} \varphi_{x, \lambda}^2 \right) \\
& \leq \, c_4 \left( |w|^q  +  |w|^{q-(q-3)_+}  U_{x, \lambda}^{(q-3)_+}  + U_{x, \lambda}^{q-2} \varphi_{x, \lambda}\, |w| +  U_{x, \lambda}^{q-2} \varphi_{x, \lambda}^2 \right).
\end{align*}
In the last inequality we used \eqref{eq-rey-2} to simplify the form of the remainder terms. Now we use the identity
$$
N\, (N-2) \int_\Omega U_{x, \lambda}^{q-1} w  \, dy = \int_\Omega \nabla U_{x, \lambda} \cdot \nabla w \, dy  = \int_\Omega \nabla P U_{x, \lambda} \cdot \nabla w \, dy  =0,
$$
which follows from \eqref{eq-pu},  \eqref{u-eq} and $w \in T_{x,\lambda}^\bot$, and the fact that  $\int_\Omega |w|^q \to 0$, which follows from \eqref{lim-eps} and the Sobolev inequality. 
Therefore, with the help of the H\"older inequality, we find 
\begin{align*}
& \quad \, \left | \int_\Omega \left( \, |\alpha|^{-q} |u_\eps|^q - P U_{x, \lambda}^q - \frac{q(q-1)}{2}\,  U_{x, \lambda}^{q-2} w^2 \right) dy  \, \right | \\
& \leq \, c_4 \Bigg[  \int_\Omega  |w|^q \, dy
 + \left(\int_\Omega  |w|^q \, dy \right)^{\frac{q-(q-3)_+}{q}} \left( \int_\Omega  U_{x, \lambda}^{q} \, dy \right)^{\frac{(q-3)_+}{q}}  \\
& \quad \qquad +  \left( \int_\Omega U_{x, \lambda}^{\frac{q(q-2)}{q-1}}\, \varphi_{x, \lambda}^{\frac{q}{q-1}} \, dy  \right)^{\frac{q-1}{q}}\, \left(\int_\Omega  |w|^q \, dy  \right)^{\frac{1}{q}} + \int_\Omega U_{x, \lambda}^{q-2}\, \varphi_{x, \lambda}^2 \, dy   \Bigg] \\
& \leq \, c_5 \Bigg[ \left(\int_\Omega  |\nabla w|^2 \, dy \right)^{\frac{q-(q-3)_+}{2}}  +  \left( \int_\Omega U_{x, \lambda}^{\frac{q(q-2)}{q-1}}\, \varphi_{x, \lambda}^{\frac{q}{q-1}} \, dy  \right)^{\frac{q-1}{q}}\, \left(\int_\Omega  |\nabla w|^2 \, dy  \right)^{\frac{1}{2}} + \int_\Omega U_{x, \lambda}^{q-2}\, \varphi_{x, \lambda}^2 \, dy   \Bigg] \,.
\end{align*}
In the last inequality, we used the Sobolev inequality for $w$ and the \eqref{lim-eps} for $w$, together with
$$
\int_\Omega  U_{x, \lambda}^{q} \, dy \  \leq \ \int_{\R^N}  U_{x, \lambda}^{q} \, dy  \ =\  \left( \frac{S_N}{N(N-2)}\right) ^{\frac{q}{q-2}} \,.
$$
It follows from Lemma \ref{lem-tech} and \eqref{lim-eps} that
\begin{align*} 
\left( \int_\Omega U_{x, \lambda}^{\frac{q(q-2)}{q-1}}\, \varphi_{x, \lambda}^{\frac{q}{q-1}} \, dy \right)^{\frac{q-1}{q}} \ & = \ 
o\left((d \lambda)^{\frac{2-N}{2}}\right)\, ,  \\
 \int_\Omega U_{x, \lambda}^{q-2}\, \varphi_{x, \lambda}^2 \, dy  \ & = \ o\left((d\, \lambda)^{2-N}\right).
\end{align*}
Thus, we conclude that, as $\eps \to 0$,
\begin{align*}
& \left | \int_\Omega \left( \, |\alpha|^{-q} |u_\eps|^q - \, P U_{x, \lambda}^q - \frac{q(q-1)}{2}\,  U_{x, \lambda}^{q-2} w^2 \right) \, dy \, \right |  = 
o\left(\int_\Omega |\nabla w|^2 \, dy \, + (\lambda d)^{2-N}\right).
\end{align*}

\smallskip

\emph{Proof of \eqref{str-3}.    }
We write 
\begin{equation} \label{eq-V-2}
|\alpha|^{-2} \int_\Omega V u_\eps^2 \, dy  =  \int_\Omega V\, P U_{x, \lambda}^2 \, dy  + 2\,  \int_\Omega V \, P U_{x, \lambda}\,  w \, dy  + \int_\Omega V  w^2 \, dy \, .
\end{equation}
By the H\"older and Sobolev inequalities we have 
$$
\left |\, \int_\Omega V  w^2 \, dy \, \right | \leq \left(\int_\Omega |V|^{\frac N2} \, dy  \right)^{\frac 2N} \left(\int_\Omega |w|^q \, dy \right)^{\frac 2q} \leq S_N^{-1} \left(\int_\Omega |V|^{\frac N2}\, dy  \right)^{\frac 2N} \int_\Omega |\nabla w|^2 \, dy  \, , 
$$
and
\begin{align*}
\left |\, \int_\Omega V P U_{x, \lambda}\,  w \, dy \, \right |& \leq \left(  \int_\Omega |V| \, P U_{x, \lambda}^2 \, dy  \right)^{\frac 12}\, \left(\int_\Omega |V|\, w^2 \, dy  \right)^{\frac 12} \\
& \leq S_N^{-1/2} \left(  \int_\Omega |V| \, P U_{x, \lambda}^2 \, dy   \right)^{\frac 12}\, \left(\int_\Omega |V|^{\frac N2} \, dy \right)^{\frac 1n} \left (\int_\Omega |\nabla w|^2 \, dy \right)^{\frac 12}.
\end{align*}
Hence \eqref{str-3} follows by inserting these estimates into \eqref{eq-V-2}.
\end{proof}

\section{\bf Proof of the main results}
\label{sec pfmain}

We now deduce Theorems \ref{thm expansion} and \ref{thm-minimizers} from Proposition \ref{prop-lowerbd}. To do so, we make crucial use of the following coercivity bound  proved in \cite[Appendix D]{rey2}.

\begin{prop} \label{prop-rey}
For all $x \in \Omega$, $\lambda > 0$ and $v \in T_{x, \lambda}^\perp$, one has
\begin{equation} \label{eq-rey}
\int_\Omega |\nabla v|^2 \, dy  - N(N+2) \, \int_\Omega U_{x, \lambda}^{q-2}\,v^2 \, dy  \ \geq \, \frac{4}{N+4} \int_\Omega |\nabla v|^2\, dy \, .
\end{equation}
\end{prop}

\begin{cor}
\label{cor-lowerbd}
For all $\eps > 0$ small enough, we have, if $N\geq 5$,
\begin{align} 
0 \geq (1 + o(1)) (S_N - S(\epsilon V)) &+  \left( \frac{S_N}{N(N-2)}\right)^{\frac{2}{2-q}} \left( \frac{N(N-2)\, a_N \, \phi(x)}{\lambda^{N-2}} +b_N\, \eps\, \frac{V(x)}{\lambda^{2}} \right) \nonumber \\
& + c \int_\Omega |\nabla w|^2 \, dy +  o((\lambda d)^{2-N}) + o(\epsilon \lambda^{-2}) \label{lowerbd-cor}
\end{align}
and, if $N=4$,
\begin{align}
0 \geq (1 + o(1)) (S_4 - S(\epsilon V)) &+ \frac{8}{S_4} \left( \frac{8 a_4 \phi(x)}{\lambda^{2}} + b_4 V(x) \frac{ \epsilon \log \lambda }{\lambda^2} \right) \nonumber \\
& + c \int_\Omega |\nabla w|^2 \, dy +  o((\lambda d)^{-2}) + o(\epsilon  \lambda^{-2} \log \lambda) \,.  \label{lowerbd-cor-N4}
\end{align}
\end{cor}

\begin{proof}
Firstly, it follows directly from \eqref{eq-rey} and the definition of $I[w]$ in \eqref{definition I} that there is a $c >0$ such that for all $\epsilon> 0$ small enough, we have 
\begin{equation}
\label{coercivity of I} I[w] \geq 4c \int_\Omega |\nabla w|^2 \, dy \,. 
\end{equation}  
Using Proposition \ref{prop-lowerbd} and \eqref{coercivity of I} it follows that for $\eps$ small enough one has
\begin{align*}
\mathcal S_{\eps V} [u_\epsilon] \geq \mathcal S_{\eps V} [PU_{x, \lambda}] + 2c \int_\Omega |\nabla w|^2 \, dy +  \mathcal{O}\left(\eps\,  \sqrt{\int_\Omega |\nabla w|^2\, dy }\ \sqrt{ \int_\Omega |V|\, P U_{x, \lambda}^2\, dy  } \  \right) 
 +  o\left( (\lambda d)^{2-N}\right). 
\end{align*}
Since
$$
\eps\,  \sqrt{\int_\Omega |\nabla w|^2\, dy }\ \sqrt{ \int_\Omega |V| \, P U_{x, \lambda}^2\, dy } \ \leq \ c \int_\Omega |\nabla w|^2\, dy  + \frac{ \eps^2}{4c}\, \int_\Omega |V|\, P U_{x, \lambda}^2\, dy \, ,
$$
this further implies that for $\eps>0$ small enough 
\begin{align*}
\mathcal S_{\eps V} [u_\epsilon] & \geq \mathcal S_{\eps V} [PU_{x, \lambda}] + c \int_\Omega |\nabla w|^2 \, dy 
+  \mathcal{O}\left(\eps^2 \int_\Omega |V|\, P U_{x, \lambda}^2 \, dy   \right)   +  o\left( (\lambda d)^{2-N}\right).
\end{align*}
Using \eqref{exp-epsVPU} for the potential term and recalling \eqref{lim-eps}, we obtain 
\begin{align*}
\mathcal S_{\eps V} [u_\epsilon] \geq 
\begin{cases} 
 \mathcal S_{\eps V} [PU_{x, \lambda}] + c \int_\Omega |\nabla w|^2 \, dy   + 
o (\eps \lambda^{-2}) + o((\lambda d)^{2-N}), & N \geq 5, \\
 \mathcal S_{\eps V} [PU_{x, \lambda}] + c \int_\Omega |\nabla w|^2 \, dy   + 
o (\eps \lambda^{-2} \log \lambda) + o((\lambda d)^{-2}), & N = 4. 
\end{cases}
\end{align*}

Now the fact that $S_N - \mathcal S_{\eps V} [u_\epsilon] = (1 + o(1)) (S_N - S(\epsilon V))$ by \eqref{appr-min}, together with the expansion of $\mathcal S_{\eps V} [PU_{x, \lambda}]$ from Theorem \ref{thm expansion PU}, implies the claimed bounds \eqref{lowerbd-cor} and \eqref{lowerbd-cor-N4}. 
\end{proof}

In the next lemma, we prove that the limit point $x_0$ lies in the set $\mathcal N(V)$. 

\begin{lem}
\label{lemma bdry conc}
We have $x_0 \in \mathcal N(V)$. In particular, $d^{-1} = \mathcal O(1)$ as $\eps \to 0$ and $x \in \mathcal N(V)$ for $\eps$ small enough. 
\end{lem}

\begin{proof}
We first treat the case $N \geq 5$. In \eqref{lowerbd-cor}, we drop the non-negative gradient term and write the remaining lower order terms as 
\begin{align*}
 & \quad \left( \frac{S_N}{N(N-2)}\right)^{\frac{2}{2-q}} \left( \frac{N(N-2)\, a_N \, \phi(x)}{\lambda^{N-2}} +b_N\, \eps\, \frac{V(x)}{\lambda^{2}} \right) + o((\lambda d)^{2-N}) + o(\epsilon \lambda^{-2}) \\
 &= \left( \frac{S_N}{N(N-2)}\right)^{\frac{2}{2-q}} \left( A (d\lambda)^{2-N} - B \eps (d \lambda)^{-2} \right), 
\end{align*}
where 
\begin{equation}
\label{def AB lemma} A = N(N-2)\, a_N \, \phi(x) d^{N-2} + o(1), \qquad B = - b_N V(x_0) d^2 + o(1). 
\end{equation}
Notice that since $\phi(x) \gtrsim d^{2-N}$ by \eqref{phi near bdry}, the quantity $A$ is positive and bounded away from zero. Moreover, by \eqref{lowerbd-cor} and the fact that $S(\eps V) < S_N$, which follows from Corollary \ref{cor-upperb}, we must have $B > 0$.  Optimizing in $d\lambda$ yields the lower bound 
\begin{equation}
\label{lowerbd remainders}
 A (d\lambda)^{2-N} - B \eps (d \lambda)^{-2} \geq - c A^{-\frac{2}{N-4}} B^\frac{N-2}{N-4}  \epsilon^\frac{N-2}{N-4}, 
\end{equation}
for some explicit constant $c > 0$ independent of $\eps$. 
On the other hand, by Corollary \ref{cor-upperb}, there is $\rho > 0$ such that the leading term in \eqref{lowerbd-cor} is bounded by
\begin{equation}
\label{lowerbd leading}
(1 + o(1)) (S_N - S(\eps V)) \geq \rho\, \epsilon^\frac{N-2}{N-4}
\end{equation}
for all $\eps > 0$ small enough. Plugging \eqref{lowerbd remainders} and \eqref{lowerbd leading} into \eqref{lowerbd-cor} and rearranging terms, we thus deduce that 
\begin{equation}
\label{lowerbd B} 
B \geq \rho^\frac{N-4}{N-2} A^\frac{2}{N-2} c^{-\frac{N-4}{N-2}}. 
\end{equation}
As observed above, the quantity $A$ is bounded away from zero and therefore \eqref{lowerbd B} implies that $B$ is bounded away from zero. Hence, in view of \eqref{def AB lemma}, $d$ is bounded away from zero and $V(x_0)< 0$. The fact that $x \in \mathcal N(V)$ for $\eps$ small enough is a consequence of the continuity of $V$. This completes the proof in case $N \geq 5$.

Now we consider the case $N = 4$ in a similar way. In \eqref{lowerbd-cor-N4}, we drop the non-negative gradient term and write the remaining lower order terms as 
\begin{align}
& \quad \frac{8}{S_4} \left( \frac{8 a_4 \phi(x)}{\lambda^{2}} + b_4 V(x) \frac{ \epsilon \log \lambda }{\lambda^2} \right)  +  o((\lambda d)^{-2}) + o(\epsilon  \lambda^{-2} \log \lambda) \nonumber \\
& =  \frac{8}{S_4} \left( A (d\lambda)^{-2} - B \eps (d\lambda)^{-2} \log (d\lambda) \right) , \label{A-B n4}
\end{align}
where 
\begin{equation}
\label{def AB lemma N=4} 
A = 8 a_4 \phi(x) d^{2} + o(1), \qquad B = - b_4 (V(x_0)+o(1)) d^2 (1 - \frac{\log d}{\log d\lambda})  . 
\end{equation}
Since $\phi(x) \gtrsim d(x)^{-2}$ by \eqref{phi near bdry}, the quantity $A$ is positive and bounded away from zero. 

Moreover, by \eqref{lowerbd-cor-N4} and the fact that $S(\eps V) < S_4$, we must have $B > 0$.
Optimizing \eqref{A-B n4} in $d \lambda$ yields the lower bound 
\begin{equation}
\label{lowerbd remainders n4}
A (d\lambda)^{-2} - B \eps (d\lambda)^{-2} \log (d\lambda) \geq - \frac{B \epsilon}{2e } \exp \left( -\frac{2A}{B \epsilon} \right) = - \exp \left( -\frac{2A}{B \epsilon} + \log(\frac{B \epsilon}{2e}) \right) .  
\end{equation}
On the other hand, by Corollary \ref{cor-upperb}, there is $\rho > 0$ such that the leading term in \eqref{lowerbd-cor-N4} is bounded by
\begin{equation}
\label{lowerbd leading n4} 
(1 + o(1)) (S_4 - S(\eps V)) \geq  \exp(-\frac{\rho}{\epsilon}). 
\end{equation}
Plugging \eqref{lowerbd remainders n4} and \eqref{lowerbd leading n4} into \eqref{lowerbd-cor-N4}, we thus deduce that 
\[ 0 \geq  \exp(-\frac{\rho}{\epsilon}) - \exp \left( -\frac{2A}{B \epsilon} + \log(\frac{B \epsilon}{2e}) \right)\, , \]
which leads to 
\begin{equation}
\label{AB intermediate} -\frac{2A}{B} + \epsilon \log(\frac{B \epsilon}{2e}) \geq - \rho  . 
\end{equation}
Since $\phi(x) \gtrsim d^{-2}$ by \eqref{phi near bdry}, the quantity $A$ is bounded away from zero and moreover $B$ is bounded. Using this fact, the left hand side of \eqref{AB intermediate} can be written as
\[ -\frac{2A}{B} (1 - \frac{B \epsilon \log B}{2 A}) + \epsilon \log \frac{\epsilon}{2e} =  - \frac{2A}{B} \left(1 + o(1)\right) + o(1). \]
Together with \eqref{AB intermediate}, this easily implies, if $\epsilon > 0$ is small enough, that
\[ B \geq \frac{A}{\rho}. \]
As before, in view of \eqref{def AB lemma N=4}, we deduce that $d$ is bounded away from zero and that $V(x_0) < 0$. The fact that $x \in \mathcal N(V)$ for $\eps$ small enough is again a consequence of the continuity of $V$.
\end{proof}

\begin{proof}
[Proof of Theorem \ref{thm expansion}]
We first treat the case $N \geq 5$. In view of Lemma \ref{lemma bdry conc}, the lower bound \eqref{lowerbd-cor} can be written as (upon dropping the non-negative gradient term) 
\begin{align*}
0 & \geq (1 + o(1)) (S_N - S(\epsilon V)) +  \!\left( \frac{S_N}{N(N-2)}\right)^{\frac{2}{2-q}} \!\!\left( \!\frac{N(N-2)\, a_N \, (\phi(x_0)+ o(1))}{\lambda^{N-2}} +b_N\, \eps\, \frac{V(x_0)+o(1)}{\lambda^{2}} \! \right) \\
& \geq (1 + o(1)) (S_N - S(\epsilon V)) - C_N (\phi(x_0)+o(1))^{-\frac{2}{N-4}}\ |V(x_0) + o(1)|^{\frac{N-2}{N-4}} \eps^\frac{N-2}{N-4} 
\end{align*}
by optimization in $\lambda$. Therefore 
\[ S(\epsilon V) \geq S_N - C_N \phi(x_0)^{-\frac{2}{N-4}}\ |V(x_0)|^{\frac{N-2}{N-4}} \eps^\frac{N-2}{N-4} + o(\eps^\frac{N-2}{N-4}) \geq S_N - C_N \sigma_N(\Omega, V) \eps^\frac{N-2}{N-4} + o(\eps^\frac{N-2}{N-4}), \]
where the last inequality uses the fact that $x_0 \in \mathcal N(V)$ by Lemma \ref{lemma bdry conc}. 

Since the matching upper bound has already been proved in Theorem \ref{thm expansion PU}, the proof in case $N \geq 5$ is complete. 

Similarly, we can handle the case $N = 4$. In view of Lemma \ref{lemma bdry conc}, the lower bound \eqref{lowerbd-cor-N4} can be written as (upon dropping the non-negative gradient term) 
\begin{align*}
0 &\geq (1 + o(1)) (S_4 - S(\epsilon V)) + \frac{8}{S_4} \left( \frac{8 a_4 (\phi(x_0)+o(1))}{\lambda^{2}} + b_4 (V(x_0) + o(1)) \frac{ \epsilon \log \lambda }{\lambda^2} \right) \\
& \geq (1 + o(1)) (S_4 - S(\epsilon V)) - \frac{4 b_4}{e S_4} \epsilon |V(x_0) + o(1)| \exp \left(- \frac{4 (\phi(x_0) + o(1))}{\epsilon |V(x_0) + o(1)|} \right) 
\end{align*}
by optimization in $\lambda$. Therefore 
\[ S(\epsilon V) \geq S_4 -  \exp\left( - \frac 4\epsilon \left(1 +o(1)\right) \frac{\phi(x_0)}{|V(x_0)|}  \right) \geq  S_4 -  \exp\left( - \frac 4\epsilon \left(1 +o(1)\right) \sigma_4(\Omega,V)^{-1} \right)\, , \]
where the last inequality uses the fact that $x_0 \in \mathcal N(V)$ by Lemma \ref{lemma bdry conc}. 

Since the matching upper bound has already been proved in Theorem \ref{thm expansion PU}, the proof in case $N = 4$ is complete. 
\end{proof}

\begin{proof}
[Proof of Theorem \ref{thm-minimizers}]
We start again with the bounds from Corollary \ref{cor-lowerbd}, but this time we need to take into account the various nonnegative remainder terms more carefully. 

\emph{Proof for $N \geq 5$.    }  We rewrite \eqref{lowerbd-cor}, using Lemma \ref{lemma bdry conc}, as 
\begin{equation}
\label{apprminproof start} 
0  \geq (1 + o(1)) (S_N - S(\epsilon V)) - C_N (\phi(x_0)+o(1))^{-\frac{2}{N-4}}\ |V(x_0)+o(1)|^{\frac{N-2}{N-4}} \eps^\frac{N-2}{N-4} + \mathcal R
\end{equation}
with 
\[\mathcal R = \left(  \frac{A_\eps}{\lambda^{N-2}} - B_\eps \frac{\epsilon}{\lambda^2} + C_N A_\eps^{-\frac{2}{N-4}} B_\eps^\frac{N-2}{N-4} \epsilon^\frac{N-2}{N-4}  \right)   + c \int_\Omega |\nabla w|^2 \, dy \, ,  \]
where we have set 
\[ A_\eps =  \left( \frac{S_N}{N(N-2)}\right)^{\frac{2}{2-q}} \left( N(N-2)\, a_N \, (\phi(x_0)+ o(1)) \right) , \quad B_\eps =  \left( \frac{S_N}{N(N-2)}\right)^{\frac{2}{2-q}}   b_N\, (  V(x_0)+o(1) )  \, . \]
Notice that both summands of $\mathcal R$ are separately nonnegative. 
Inserting the upper bound from Corollary \ref{cor-upperb} into \eqref{apprminproof start}, we get 
\[ 0 \geq C_N \left(\sigma_N(\Omega, V) - \phi(x_0)^{-\frac{2}{N-4}}\ |V(x_0)|^{\frac{N-2}{N-4}} \right)  \epsilon^\frac{N-2}{N-4} + \mathcal R +  o(\epsilon^\frac{N-2}{N-4})\, . \]
Since each one of the first two summands on the right hand side is nonnegative, we deduce that 
\[ \phi(x_0)^{-\frac{2}{N-4}}\ |V(x_0)|^{\frac{N-2}{N-4}} = \sup_{x \in \mathcal N(V)} \phi(x)^{-\frac{2}{N-4}}\ |V(x)|^{\frac{N-2}{N-4}} = \sigma_N(\Omega, V)\]
and
\begin{equation}
\label{R is small} 
 \mathcal R = o(\epsilon^\frac{N-2}{N-4}). 
\end{equation}
In particular, \eqref{R is small} implies that 
\begin{equation}
\label{apprminproof w} \|\nabla w\|_2^2 = o(\epsilon^\frac{N-2}{N-4}). 
\end{equation}
Denote by 
\[ \lambda_0(\eps) = \left(\frac{(N-2)A_\eps}{2B_\eps } \right)^\frac{1}{N-4}  \eps^{\frac{1}{4-N}} \]
the unique value of $\lambda$ for which the first summand of $\mathcal R$ vanishes.
Using Lemma \ref{lem-taylor}, the bound \eqref{R is small} implies that 
\[  \eps (\lambda^{-1} - \lambda_0(\eps)^{-1})^2 = o(\epsilon^\frac{N-2}{N-4}),  \]
which is equivalent to 
\begin{equation}
\label{apprminproof lambda} 
\lambda = \lambda_0(\eps) + o(\eps^{-\frac{1}{N-4}}) = \left(\frac{N\, (N-2)^2\, a_N \, \phi(x_0)}{2 \,b_N\,  |V(x_0)|}\right)^{\frac{1}{N-4}}\, \eps^{-\frac{1}{N-4}} + o(\eps^{-\frac{1}{N-4}}) . 
\end{equation} 
Finally, to obtain the asymptotics of $\alpha$, by \eqref{str-2}, \eqref{appr-min}, \eqref{exp-PUq} and \eqref{apprminproof w}, we have that 
\begin{equation}
\label{apprminproof alpha start} |\alpha|^{-q} \left( \frac{S_N}{N(N-2)}\right)^{\frac{q}{q-2}} = \left( \frac{S_N}{N(N-2)}\right)^{\frac{q}{q-2}} - q a_N \lambda^{2-N} \phi(x_0) +\frac{q(q-1)}{2}\,  \int_\Omega U_{x, \lambda}^{q-2}\, w^2 \, dy +  o(\lambda^{2-N})\, .
\end{equation}
Moreover, by Hölder and Sobolev inequalities, 
\begin{equation}
\label{apprminproof alpha bound w}
\int_\Omega U_{x, \lambda}^{q-2} w^2 \, dy  \lesssim \|\nabla w\|^2. 
\end{equation}
We easily conclude from \eqref{apprminproof w}--\eqref{apprminproof alpha bound w}  that 
\[ |\alpha| = 1 + D_N \sigma_N(\Omega, V) \eps^\frac{N-2}{N-4} + o(\eps^\frac{N-2}{N-4}) \]
with $D_N$ given in \eqref{dn}. This completes the proof of Theorem \ref{thm-minimizers} in the case $N \geq 5$.

\emph{Proof for $N = 4$.    }
We rewrite \eqref{lowerbd-cor-N4}, using Lemma \ref{lemma bdry conc}, as
\begin{equation}
\label{apprminproof start n4} 
0  \geq (1 + o(1)) (S_4 - S(\epsilon V))
- \frac{B_\eps \epsilon}{2 e} \exp\left(- \frac{2A_\eps}{B_\eps \epsilon} \right)
+  \mathcal R
\end{equation}
with
$$
\mathcal R = \left( \frac{A_\eps}{\lambda^{2}} - B_\eps \frac{\epsilon \log \lambda}{\lambda^2} + \frac{B_\eps \epsilon}{2 e} \exp\left(- \frac{2A_\eps}{B_\eps \epsilon} \right) \right)
+ c \int_\Omega |\nabla w|^2 \, dy,  
$$
where we have set 
\[ A_\eps = \frac{64}{S_4}  a_4 (\phi(x_0) +o(1)), \qquad B_\eps = \frac{8}{S_4} b_4 |V(x_0)+o(1)| \, . \]
Notice that both summands of $\mathcal R$ are separately nonnegative. Inserting the upper bound from Corollary \ref{cor-upperb} into \eqref{apprminproof start n4}, we get
\begin{equation}
\label{eq:proofn4}
0 \geq (1 + o(1))\exp\left( - \frac 4\epsilon \left(1 +o(1)\right) \sigma_4(\Omega,V)^{-1} \right) - \frac{B_\eps \epsilon}{2 e} \exp\left(- \frac{2A_\eps}{B_\eps \epsilon} \right)
+  \mathcal R \,.
\end{equation}
Dropping the nonnegative term $\mathcal R$ from the right side and taking the logarithm of the resulting inequality, we obtain
$$
- \frac{2A_\eps}{B_\eps \epsilon} + \log\frac{B_\eps \epsilon}{2 e} \geq - \frac 4\epsilon \left(1 +o(1)\right) \sigma_4(\Omega,V)^{-1} + \log(1+o(1)) \,.
$$
Multiplying by $\epsilon$ and passing to the limit we infer, since $a_4/b_4 = 1/4$,
$$
- \frac{\phi(x_0)}{|V(x_0)|} \geq - \sigma_4(\Omega,V)^{-1} \,.
$$
By definition of $\sigma_4(\Omega,V)$, this implies
\begin{equation}
\label{eq:proofn41}
\frac{|V(x_0)|}{\phi(x_0)} = \sigma_4(\Omega,V) \,,
\end{equation}
as claimed.

With this information at hand, we return to \eqref{eq:proofn4} and drop the nonnegative first term on the right side to infer that
$$
\mathcal R \leq \frac{B_\eps \epsilon}{2 e} \exp\left(- \frac{2A_\eps}{B_\eps \epsilon} \right).
$$
Keeping only the second term in the definition of $\mathcal R$ and using \eqref{eq:proofn41} we deduce, in particular, that
\begin{equation}
\label{w is small n4}
 \|\nabla w \|_2^2 \leq \exp\left( - \frac 4\epsilon \left(1 +o(1)\right) \sigma_4(\Omega,V)^{-1} \right). 
 \end{equation}

We now keep only the first term in the definition of $\mathcal R$ and obtain from \eqref{eq:proofn4}, multiplied by $(2e/(B_\epsilon \epsilon))\exp(2A_\epsilon/(B_\epsilon \epsilon))$,
\begin{align*}
1 - (1+o(1)) \frac{2 e}{B_\eps \epsilon} \exp\left( \frac{2A_\eps}{B_\eps \epsilon} - \frac 4\epsilon \left(1 +o(1)\right) \sigma_4(\Omega,V)^{-1} \right)
& \geq \frac{2 e}{B_\eps \epsilon} \exp\left(\frac{2A_\eps}{B_\eps \epsilon} \right)
\mathcal R \\
& \geq \frac{2 e}{B_\eps \epsilon} \exp\left(\frac{2A_\eps}{B_\eps \epsilon} \right) \left( \frac{A_\epsilon}{\lambda^2} - B_\epsilon \frac{\epsilon \log\lambda}{\lambda^2} \right) + 1 \\
& = 1+ y\, e^{y+1}
\end{align*}
with $y = \frac{2}{B_\epsilon \epsilon} (A_\epsilon - \eps B_\epsilon \log \lambda)$.  In view of \eqref{eq:proofn41} and \eqref{eq-revised} we have
$$
(1+o(1)) \frac{2 e}{B_\eps \epsilon} \exp\left( \frac{2A_\eps}{B_\eps \epsilon} - \frac 4\epsilon \left(1 +o(1)\right) \sigma_4(\Omega,V)^{-1} \right) = \exp \left( o \left( \frac{1}{\epsilon} \right)\right),
$$
and therefore
$$
- \exp \left( o \left( \frac{1}{\epsilon} \right)\right) \geq y \,e^{y+1} \,.
$$
This implies
$$
0 < -y \leq o \left( \frac{1}{\epsilon} \right),
$$
which is the same as
$$
\frac{A_\epsilon}{B_\epsilon \epsilon} < \log \lambda \leq \frac{A_\epsilon}{B_\epsilon \epsilon} + o \left( \frac{1}{\epsilon} \right).
$$
Recalling \eqref{eq:proofn41} we obtain
\begin{equation}
\label{lambda asymptotics proof}
\lambda = \exp\left( - \frac 2\epsilon \left(1 +o(1)\right) \sigma_4(\Omega,V)^{-1} \right),
\end{equation}
as claimed.

Finally, to obtain the asymptotics of $\alpha$, we deduce from \eqref{apprminproof alpha start} and \eqref{apprminproof alpha bound w}, together with the bounds \eqref{w is small n4} and \eqref{lambda asymptotics proof}, that 
\[ |\alpha| = 1 +  \exp\left( -  \frac 4\epsilon \left(1 +o(1)\right) \sigma_4(\Omega,V)^{-1} \right). \]
This completes the proof of Theorem \ref{thm-minimizers} in the case $N = 4$. 
\end{proof}

\appendix

\section{Auxiliary results}
\label{sec-app}
The proof of the following lemma is similar to the computation in \cite[Appendix A]{rey2}. We provide here details for the sake of completeness. 
\begin{lem} \label{lem-tech}
Let $x = x_\lambda$ be a sequence of points in $\Omega$ such that $d(x) \lambda \to\infty$. Then
\begin{equation} 
\label{eq-a} 
\left( \int_\Omega U_{x,\lambda}^{\frac{q(q-2)}{q-1}}\, \varphi_{x,\lambda}^{\frac{q}{q-1}} \, dy  \right)^{\frac{q-1}{q}}  \  = 
\begin{cases} 
\mathcal{O}\left((d(x)\, \lambda)^{\frac{-2-N}{2}}\right)  & \text{ if } N > 6,  \\
\mathcal{O}\left((d(x)\, \lambda)^{-4} \log(d(x)\lambda)\right)  & \text{ if } N = 6, \\
 \mathcal{O}\left((d(x)\, \lambda)^{2-N}\right)  & \text{ if } N =4,5 
\end{cases} 
\end{equation}
and
\begin{equation}
\label{eq-b} 
 \int_\Omega U_{x,\lambda}^{q-2}\, \varphi_{x,\lambda}^2 \, dy  \  =
\mathcal{O}\left((d(x)\, \lambda)^{-N}\right) \, .
\end{equation}
\end{lem}

\begin{proof}
We write $d = d(x)$ for short in the following proof. 

\emph{Proof of \eqref{eq-a}.    }
By equations \eqref{u-split}, \eqref{sup-f} and \eqref{sup-h},
\begin{equation} \label{holder-in}
\int_{B_d(x)} U_{x,\lambda}^{\frac{q(q-2)}{q-1}}\, \varphi_{x,\lambda}^{\frac{q}{q-1}}\, dy \, \leq \|\varphi_{x,\lambda}\|^{\frac{q}{q-1}}_{L^\infty(\Omega)}\, \int_{B_d(x)} U_{x,\lambda}^{\frac{q(q-2)}{q-1}} \, dy  = \mathcal{O}\left( (d^{2-N}\, \lambda^{\frac{2-N}{2}})^{\frac{q}{q-1}}\right) \, \int_{B_d(x)} U_{x,\lambda}^{\frac{q(q-2)}{q-1}} \, dy \, .
\end{equation}
Moreover, since $\frac{q(q-2)}{q-1}\, \frac{N-2}{2} = \frac{4N}{N+2}$, from \eqref{u-function} we obtain
\begin{align}
\int_{B_d(x)} U_{x,\lambda}^{\frac{q(q-2)}{q-1}} \, dy  & = \mathcal{O}\left( \lambda^{\frac{4N}{N+2}}\right)\, \int_0^d \frac{r^{N-1}\, dr}{(1+\lambda^2\, r^2)^{\frac{4N}{N+2}}} = \mathcal{O}\left( \lambda^{\frac{2N-N^2}{N+2}}\right)\, \int_0^{\lambda d} \frac{t^{N-1}\, dr}{(1+ t^2)^{\frac{4N}{N+2}}} \nonumber \\
& = \mathcal{O}\left( \lambda^{\frac{2N-N^2}{N+2}}\right)\, \left(\int_1^{\lambda d} t^{\frac{N(N-6)}{N+2}} \, t^{-1} \, dt +\mathcal{O}(1)\right). \label{1-ball-in}
\end{align}
If $N>6$, then 
$$
\int_1^{\lambda d} t^{\frac{N(N-6)}{N+2}} \, t^{-1} \, dt = \mathcal{O}\left( ( d\, \lambda)^{\frac{N(N-6)}{N+2}} \right).
$$
If $N = 6$, then 
$$
\int_1^{\lambda d} t^{\frac{N(N-6)}{N+2}} \, t^{-1} \, dt = \mathcal{O}\left( \log ( d\, \lambda) \right)
$$
and if $N = 4,5$, then 
$$
\int_1^{\lambda d} t^{\frac{N(N-6)}{N+2}} \, t^{-1} \, dt = \mathcal{O}\left(1 \right)
$$
This gives the bound claimed in \eqref{eq-a} in each case, provided we can bound the integral on the complement $\Omega \setminus B_d(x)$. On this region, we have by H\"older 
\begin{align*}
\left( \int_{\Omega \setminus B_d(x)} U_{x,\lambda}^{\frac{q(q-2)}{q-1}}\, \varphi_{x,\lambda}^{\frac{q}{q-1}}\, dy  \right)^{\frac{q-1}{q}} &
\leq \, \left( \int_\Omega \varphi_{x,\lambda}^{\frac{2N}{N-2}} \, dy  \right)^{\frac{N-2}{2N}}\,  \left( \int_{\R^N\setminus B_ d(x)}  U_{x,\lambda}^{\frac{2N}{N-2}} \, dy  \right)^{\frac 2N} \\
& = \mathcal{O}\left((d\, \lambda)^{\frac{2-N}{2}} \right)\,   \left( \int_{\R^N\setminus B_ d(x)}  U_{x,\lambda}^{\frac{2N}{N-2}} \, dy  \right)^{\frac 2N} \\
& = \mathcal{O}\left((d\, \lambda)^{\frac{2-N}{2}} \right)\,  \left( \int_{ d \lambda}^\infty  \frac{dt}{t^{N+1}} \right)^{\frac 2N} \\
& =
\mathcal{O}\left((d\, \lambda)^{\frac{2-N}{2}}\right) \, \mathcal{O}\left(( d\, \lambda)^{-2}\right),
\end{align*}
where we have used \eqref{u-function} and the fact that 
\begin{equation} \label{rey-pr1}
\left( \int_\Omega \varphi_{x,\lambda}^{\frac{2N}{N-2}} \, dy  \right)^{\frac{N-2}{2N}}\, =  \mathcal{O}\left((d\, \lambda)^{\frac{2-N}{2}} \right)
\end{equation}
by \cite[Prop.~1(c)]{rey2}. Combining all the estimates, we deduce \eqref{eq-a}. 

\emph{Proof of \eqref{eq-b}.    } We split the domain of integration $\Omega$ again into $B_d(x)$ and $\Omega \setminus B_d(x)$. On $B_d(x)$, by \eqref{u-split},
\begin{align} 
&\qquad \int_{B_d(x)} U_{x,\lambda}^{q-2}\, \varphi_{x,\lambda}^2  \, dy  \leq  \|\varphi_{x,\lambda}\|^2_{L^\infty(\Omega)} \left( \int_{B_ d(x)} U_{x,\lambda}^{q-2} \, dy  \right) \nonumber \\
& = \ \mathcal{O}\left( d(x)^{4-2N}\, \lambda^{2-N}\right) \left( \lambda^{2-N}  \int_0^{ d \lambda} \frac{t^{N-1}\, dt}{(1+t^2)^2} \right) = \mathcal O ((d\lambda)^{-N}). \label{b-ball}
\end{align}
On $\Omega \setminus B_d(x)$, by Hölder and \eqref{rey-pr1}, 
\begin{align}
\label{b-out}
\int_{\Omega \setminus  B_d(x)} U_{x,\lambda}^{q-2}\, \varphi_{x,\lambda}^2  \, dy  \leq \left( \int_{\Omega} \varphi_{x,\lambda}^q \, dy  \right)^\frac{2}{q} \left( \int_{\R^N\setminus B_ d(x)}  U_{x,\lambda}^q \, dy  \right)^{\frac{q-2}{q}} = \mathcal{O}\left((d(x)\, \lambda)^{2-N}\right) \, \mathcal{O}\left(( d\, \lambda)^{-2} \right).
\end{align}
Combining \eqref{b-ball} and \eqref{b-out}, we obtain \eqref{eq-b}.
\end{proof}

\begin{lem}
\label{lem-taylor}
Let $f_\eps: (0, \infty) \to \R$ be given by
\[ f_\eps(\lambda) = \frac{A_\epsilon}{\lambda^{N-2}} - B_\epsilon \frac{\epsilon}{\lambda^2}\]
with $A_\eps, B_\eps > 0$ uniformly bounded away from 0 and $\infty$. Denote by 
\[ \lambda_0 = \lambda_0(\eps) = \left(\frac{(N-2)A_\eps}{2B_\eps } \right)^\frac{1}{N-4}  \eps^{\frac{1}{4-N}} \]
the unique global minimum of $f_\eps$. Then there is a $c_0 > 0$ such that for all $\eps > 0$ we have
\[ f_\eps(\lambda) - f_\eps(\lambda_0) \geq 
\begin{cases}
c_0 \epsilon \left(  \lambda^{-1} - \lambda_0(\eps)^{-1}\right)^2 & \text{ if } \quad (\frac{A_\epsilon}{B_\eps})^{\frac{1}{N-4}} \eps^{-\frac{1}{N-4}} \lambda^{-1} \leq 2 (\frac{2}{N-2})^\frac{1}{N-4} , \\
c_0 \epsilon^\frac{N-2}{N-4}  & \text{ if } \quad  (\frac{A_\epsilon}{B_\eps})^{\frac{1}{N-4}} \eps^{-\frac{1}{N-4}} \lambda^{-1} > 2 (\frac{2}{N-2})^\frac{1}{N-4}.
\end{cases} \]
\end{lem}

\begin{proof}
Let $F(t):= t^{N-2} - t^2$ and denote by $t_0 := (\frac{2}{N-2})^\frac{1}{N-4}$ the unique global minimum on $(0, \infty)$ of $F$. Then it is easy to see that there is $c > 0$ such that 
\[ F(t) - F(t_0) \geq 
\begin{cases}
c(t - t_0)^2 & \text{ if } \quad 0 < t  \leq 2 t_0, \\
c t_0^{N-2} & \text{ if } \quad t > 2 t_0.
\end{cases}
\]
The assertion of the lemma now follows by rescaling. Indeed, it suffices to observe that 
\[ f_\eps(\lambda) = A_\eps^{-\frac{2}{N-4}} B_\eps^\frac{N-2}{N-4} \eps^\frac{N-2}{N-4} F\left( (\frac{A_\epsilon}{B_\eps})^{\frac{1}{N-4}} \eps^{-\frac{1}{N-4}} \lambda^{-1} \right) \]
and to use the boundedness of $A_\eps$ and $B_\eps$. 
\end{proof}

\end{document}